\documentclass[a4paper,10pt,leqno]{amsart}

\usepackage{lineno,hyperref}
\modulolinenumbers[5]

\usepackage[english]{babel}
\usepackage[utf8]{inputenc}
\usepackage{amsmath,amsthm,amssymb}
\usepackage{lmodern}

\newcommand{\R}{\mathbb{R}}
\newcommand{\C}{\mathbb{C}}
\newcommand{\N}{\mathbb{N}}
\newcommand{\A}{\mathcal{A}}

\renewcommand{\L}{\mathcal{L}}

\newcommand{\B}{\mathcal{B}}
\newcommand{\const}{\mathbf c}

\renewcommand{\Re}{\operatorname{Re}}

\newcommand{\fra}{\mathfrak{a}}

\renewcommand{\A}{\mathcal{A}}
\renewcommand{\B}{\mathcal{B}}
\renewcommand{\mid}{\, \vert \,}

\newcommand{\lra}{\longrightarrow}

\theoremstyle{plain}
\newtheorem{theorem}{Theorem}[section]
\newtheorem{proposition}[theorem]{Proposition}
\newtheorem{lemma}[theorem]{Lemma}
\newtheorem{corollary}[theorem]{Corollary}
\newtheorem{remark}[theorem]{Remark}
\theoremstyle{definition}

\newtheorem{notation}[theorem]{Notation}

\theoremstyle{remark}
\begin{document}

\title{Stability for evolution equations governed by a non-autonomous form}

\author{OMAR EL-MENNAOUI}
\address{OMAR EL-MENNAOUI: Department of Mathematics,
	University Ibn Zohr, Faculty of Sciences, Agadir, Morocco}
\email{elmennaouiomar@yahoo.fr}

\author{HAFIDA LAASRI}
\address{HAFIDA LAASRI: Lehrgebiete Analysis, Fakult\"at Mathematik und Informatik, Fernuniversit\"at in Hagen, D-58084 Hagen, Germany}
\email{hafida.laasri@fernuni-hagen.de}


\thanks{}



\begin{abstract}
This paper deals with the approximation of non-autonomous evolution equations of the form
\begin{equation*}\label{Abstract equation}
\dot u(t)+\A(t)u(t)=f(t)\ \ t\in[0,T],\ \
u(0)=u_0.
\end{equation*}
where $\A(t),\ t\in [0,T]$  arise from  a non-autonomous sesquilinear forms $\fra(t;\cdot,\cdot)$ on a Hilbert space $H$ with constant domain $V\subset H.$ Assuming the existence of a sequence  $\fra_n:[0,T]\times V\times V\lra \C, n\in \N$ of non-autonomous forms such that the associated Cauchy problem has $L^2$-maximal regularity in $H$ and $a_n(t,u,v)$ converges to $a(t,u,v)$ as $n\to \infty,$ then among others we show under additional assumptions that the limit problem has $L^2$-maximal regularity. Further we show that the convergence is uniformly on the initial data $u_0$ and the inhomogeneity $f.$
\end{abstract}


\maketitle
\section*{Introduction\label{s1}}
\noindent Throughout this paper $H,V$ are two separable  Hilbert spaces over $\mathbb K=\mathbb C.$ We denote by $(\cdot \mid \cdot)_V$ the scalar product and  $\|\cdot\|_V$ the norm
on $V$ and by $(\cdot \mid \cdot), \|\cdot\|$ the corresponding quantities in $H.$ Moreover, we assume that $V$ is densely and continuously embedded into $H.$ Let $V'$ denote  the antidual of $V$ and $\langle \cdot, \cdot \rangle$ the duality between $V'$ and $V$. As usual, by identifying $H$ with  $H',$ we have   $V\hookrightarrow H\cong H'\hookrightarrow V'$ with continuous and dense embedding. Let $T>0.$ Let $\fra:[0,T]\times V\times V\lra C$ be a \textit {closed non-autonomous sesquilinear form}, i.e., $\fra(.,u,v)$ is measurable for all $u,v\in V,$ and $\fra(t;\cdot,\cdot)$ is a sesquilinear form with
\begin{equation}\label{eq:continuity-nonaut-introduction}
| \fra(t;u,v) | \le M \|u\|_V \|v\|_V, \text{ and } \ \Re~ \fra(t;u,u) +\beta\|u\|_H^2\ge \alpha \|u\|^2_V
\end{equation}
for all $t\in[0,T], v,u\in V$ and for some constants $\beta\in \R, \alpha, M>0.$ By Lax-Milgram Theorem, for each $t\in[0,T]$ there exists an isomorphism $\A(t):V\to V^\prime$ such that
\[\langle \A(t) u, v \rangle = \fra(t;u,v), \qquad \text{}  u,v \in V.\] We call $\A(t)$ the operator associated with $\fra(t,\cdot,\cdot)$ on $V^\prime.$ Seen as an unbounded operator on $V'$ with domain $D(\A(t)) = V,$ the
operator $- \A(t)$ generates a holomorphic $C_0-$semigroup $\mathcal T$ on $V'.$ Further, we denote by $A(t)$ the part
of $\A(t)$ on $H$; i.e.,\
\begin{align*}
D(A(t)) := {}& \{ u\in V : \A(t) u \in H \}\\
A(t) u := {}& \A(t)u, \quad \text{ for } u\in D(A(t)).
\end{align*}
It is a known fact that  $-A(t)$ generates a holomorphic $C_0$-semigroup  $T$ on $H$  and $T=\mathcal T_{\mid H}$ is the
restriction of the semigroup generated by $-\A$ to $H.$ Then $A(t)$ is the operator \textit{induced} by $\fra(t,.,.)$  on $H.$
See, e.g., \cite{Ar06},\cite{Ka},  \cite[Chap.\ 2]{Tan79} and \cite{Bre11}. 
\medskip

\par\noindent  Consider the non-autonomous Cauchy problem 
\medskip
\begin{equation}\label{Abstract Cauchy problem 0}
\dot{u} (t)+\A(t)u(t)=f(t),\quad{a.e} \ \text on \quad [0,T], \quad u(0)=u_0.
\end{equation}
Then the following \textit{$L^2$-maximal regularity in $V'$} result has been proved by J. L. Lions on 1961:

\begin{theorem}(\textrm{Lions} 1961)\label{wellposedness in V'} The non-autonomous Cauchy problem (\ref{Abstract Cauchy problem 0}) has $L^2$-\textit{maximal regularity in} $V',$ i.e., for given $f\in L^2(0,T;V')$ and $u_0\in H,$ (\ref{Abstract Cauchy problem 0}) has a unique solution $u$ in $MR_2(V,V'):=L^2(0,T;V)\cap H^1(0,T;V').$ Moreover, there exists a constant $c_0>0$ depending only on $\alpha, \beta, M$ and $c_H$ such that 
	\begin{equation}\label{uniform estimation V'}
	\|u\|_{MR_2(V,V')}\leq c_0\big[\|f\|_{L^2(0,T;V')}+\|u_0\|_H,\big]
	\end{equation}
	where $c_H$ is the continuous embedding constant of $V$ into $H.$
\end{theorem}
Lions proved this result in \cite{Lio61} (see also \cite[Chapter 3]{Tho03}) using a  representation theorem
of linear functionals due to him self and  usually known in the literature as  \textit{Lions's representation Theorem} and
using  Galerkin's method in  \cite[ XVIII
Chapter 3, p. 620]{DL88}. We refer  \cite[Section 5.5]{Tan79} and \cite{LASA14} for other proofs. The theorem of Lions requires  only the measurability of  $t\mapsto \fra(t;u,v)$  for all $u,v\in V$. However, in applications to boundary problems maximal regularity in $V'$ is not sufficient because it is only the part $A(t)$ of $\A(t)$ in $H$ that realizes the boundary conditions in question. Precisely one is more interested on \textit{$L^2$-maximal regularity in $H,$} i.e., the solution $u$ of (\ref{Abstract Cauchy problem 0}) belong to $H^1(0,T;H)$ if $f\in L^2(0,T; H)$ and $u_0\in V.$ The problem of $L^2$-maximal regularity in $H$ was initiated  by Lions in \cite[p.\ 68]{Lio61} for $u_0=0$ and $\fra$ is symmetric. In general, we have to impose more regularity on the form $\fra$ then measurability of the form is not sufficient \cite{D1,ADF17}. However, under additional regularity assumptions on the form $\fra,$   the
initial value $u_0$ and the inhomogeneity $f,$ some positive results were already done by Lions in
\cite[p.~68, p.~94, ]{Lio61}, \cite[Theorem~1.1, p.~129]{Lio61} and
\cite[Theorem~5.1, p.~138]{Lio61} and by Bardos \cite{Bar71}. More recently, this
problem has been studied  with some progress and different
approaches \cite{ADLO14, Ar-Mo15, O15, D2, OH14, OS10, ELLA15, Di-Za16, Fa17}. Results on multiplicative perturbation are established in \cite{ADLO14, D2, AuJaLa14}. See also the recent review paper \cite{ADF17} for more details and references.
\medskip

Exploiting the ideas and proofs of a recent result of Arendt and Monniaux \cite{Ar-Mo15} we study in this paper  stability and  the uniform approximation of the non-autonomous Cauchy problems (\ref{Abstract Cauchy problem 0}). More precisely, assume that there exists a sequence  $\fra_n:[0,T]\times V\times V\lra \C$ of non-autonomous forms such that Cauchy problem  
\begin{equation}\label{approximation problem 0}
\dot{u}_n (t)+\A_n(t)u_n(t)=f(t), \quad u(0)=u_0
\end{equation}
associated with $\fra_n$ has $L^2$-maximal regularity in $H$ and $a_n(t,u,v)$ converges to $a(t,u,v)$ as $n\to \infty.$ Then our aim is to study weather $L^2$-maximal regularity is iterated by the limit problem (\ref{Abstract Cauchy problem 0}) and weather the sequence  $(u_n)_{n\in\N}$ of solutions of (\ref{approximation problem 0}) converges uniformly on $u_0$ and $f$ to the the solution of (\ref{Abstract Cauchy problem 0}). 
Let  $0< \gamma< 1.$ Let $\omega_n:[0,T]\longrightarrow [0,+\infty), n\in\N,$ be a sequence of non-decreasing continuous function and  let $(d_n)_{n\in\N}$ be a zero real sequence such that
\begin{align*}
 &|\fra_n(t,u,v)-\fra(t,u,v)| \le d_n \Vert u\Vert_{V} \Vert v\Vert_{V_\gamma}, \ t\in[0,T], \ u,v\in V,\\	
\\&|\fra_n(t,u,v)-\fra_n(s,u,v)| \le\omega_n(|t-s|) \Vert u\Vert_{V} \Vert v\Vert_{V_\gamma}, \ t\in[0,T], \ u,v\in V,\\
\\& \displaystyle\sup_{t\in[0,T],n\in \N} \frac{\omega_n(t)}{t^{\gamma/2}}<\infty,  \text{ and }  \  \ \displaystyle\sup_{n\in \N}\int_0^T\frac{\omega_n(t)}{t^{1+\gamma/2}} {\rm d}t<\infty\ \
\end{align*}
for all $t,s\in[0,T], n\in\N$ and for all $u,v\in V,$ where $V_\gamma:=[H,V]_\gamma$ is the complex interpolation space. Then we show in Section \ref{weak approximation H} that the limit problem (\ref{Abstract Cauchy problem 0}) has also $L^2$-maximal regularity in $H$ and the sequence  $(u_n)_{n\in\N}$ of solutions of (\ref{approximation problem 0}) converges weakly in $MR_2(V,H)$ to the solution of (\ref{Abstract Cauchy problem 0}). This convergences holds for the  strongly topology of $MR_2(V,H)$ and uniformly on $u_0$ and $f$ provided the sequence $(d_n)_{n\in\N}$ is decreasing, $\lim\limits_{n\to \infty}d_nn^{\gamma/2}=0$ and  $\lim\limits_{n\to \infty} \displaystyle\int_{0}^{1/n}\frac{\omega_n(r)}{r^{{1+\gamma/2}}}  {\rm d}r= 0,$ see Section \ref{Uniform Convergence}. Moreover we show that similar results holds on the space $C(0,T;V)$ if  $(u_n)_{n\in\N}\subset C(0,T;V).$
In the last section we provide an explicit approximation of $\fra$ that satisfies the above required hypothesis. The reader interested in examples of application is referred to above cited papers and the references therein. 
\section{Preliminary results: uniform approximation on $V'$}\label{section Uniform Cv V'}
\noindent In this section $\fra:[0,T]\times V\times V\lra C$ is a closed non-autonomous sesquilinear form. Moreover, we assume that there exist a sequence of closed non-autonomous sesquilinear forms $\fra_n:[0,T]\times V\times V\lra C$ satisfying (\ref{eq:continuity-nonaut-introduction}) with the same constants $\beta, \alpha$ and $M>0$ and a zero real sequence $(d_n)_{n\in\N}$ such that the following assumption holds:
\vskip0.5cm
${\bf (H_0)} \   |\fra(t;u,v)-\fra_n(t;u,v)| \le d_n \|u\|_V \|v\|_V, \ t\in [0,T], u,v\in V.$
\vskip0.5cm

\noindent For each $t\in[0,T]$ and $n\in\N,$ let $\A_n(t)\in\L(V,V')$ be the operator associated with $\fra_n(t;\cdot,\cdot)$ on $V'$ and consider the approximation Cauchy problems 
\begin{equation}\label{Approximation Abstract Cauchy problem}
\dot{u}_n (t)+\A_n(t)u_n(t)=f(t), \quad{a.e} \ \text on \quad [0,T], \quad u_n(0)=u_0, \ (n\in \N).
\end{equation}
Note that the maximal regularity space $MR_2(V,V')$ is continuously embedded into $C([0,T];H)$ \cite[p. 106]{Sho97}. Moreover, the result of Lions implies that $H$ coincides with the trace space, that is  \[H=Tr_2(V,V'):=\Big\{u(0)\mid u\in MR_2(V,V')\Big\}.\]

The following theorem is the main results of this section.
\begin{theorem}\label{proposition: uniform approx V'} Let $u, u_n\in MR_2(V,V')$ be the solutions of (\ref{Abstract Cauchy problem 0}) and (\ref{Approximation Abstract Cauchy problem}), respectively. Then the following inequalities 
	\[\|u_n-u\|_{MR_2(V,V')}\leq \const d_n \Big[ \|f\|_{L^2(0,T;V')}+\|u_0\|_H\Big],\ n\in\N, \qquad\text
	{ and}  \]
	\[\|u_n-u\|_{C([0,T],H)}\leq \const d_n \Big[ \|f\|_{L^2(0,T;V')}+\|u_0\|_H\Big],\ n\in\N  \]
hold for some positive constant $\const>0$  depending only on $M,\alpha,c_H$ and $T.$ The sequence  $(u_n)_{n\in\N}$ thus converges in $MR_2(V,V')\cap C([0,T],H)$ to $u$ uniformly on the data $f, u_0.$
\end{theorem}
\begin{proof} For simplicity,  we will in the sequel denote all positive constants depending on $M,\alpha,c_H$ and $T$ by $\const>0.$ In view of the above  Remark, it suffices to prove the first inequality. To that purpose, consider the unbounded linear operators $\A, \A_n$ and $\B$ with domains $D(\A)=D(\A_n)=L^2(0,T;V)$ and $D(\mathcal B)=\big\{u\in
	H^1(0,T; V'), u(0)=0\big\}$ defined by
	\[(\A f)(t)=\A_n(t)f(t),\ \ (\A_n f)(t)=\A_n(t)f(t)\ \ \hbox{   and  } \ \ (\mathcal B u)(t)=\dot{u}(t) \ \hbox{ for almost every } t\in[0,T].\]
Thus the Cauchy problem (\ref{Abstract Cauchy problem 0}), respectively (\ref{Approximation Abstract Cauchy problem}),  has $L^2-$maximal regularity in $V'$ if and only if
	the unbounded operator $\A+\B,$ respectively $\A_n+\B,$ with domain \[D(\A+\B)=D(\A_n+\B):=\{ u\in MR_2(V,V')\mid u(0)=0\}\] is invertible. 
Consider first the case where $u_0=0.$ Then we have $u=(\A+\B)^{-1}f$ and  $u_n=(\A_n+\B)^{-1}f.$ From Theorem \ref{wellposedness in V'} and $\bf(H_0)$ we have
\begin{align*}
\|u_n-u\|_{MR_2(V,V')}&=\|(\A+\B)^{-1}f-(\A_n+\B)^{-1}f\|_{MR_2(V,V')}
\\&=\|(\A+\B)^{-1}(\A_n-\A)(\A_n+\B)^{-1}f\|_{MR_2(V,V')}
\\&\leq \const d_n\|f\|_{L^2(0,T;V')}.
\end{align*}
 Let now $0\neq u_0\in H.$ Choose $\vartheta\in MR_2(V,V')$ such that $\vartheta(0)=u_0$ and $\|\vartheta\|_{MR_2(V,V')}\leq 2\|u_0\|_{H}.$ Set  $g_n:=-\dot{\vartheta}(\cdot)-\A_n(\cdot)\vartheta(\cdot)+f(\cdot)$ and
$g:=-\dot{\vartheta}(\cdot)-\A(\cdot)\vartheta(\cdot)+f(\cdot)\in L^2(0,T;V').$ Then there exist $v_n,u\in MR_2(V,V')$ such that
\begin{equation*}
\dot{v}_n(t)+\A_n(t)v_n(t)=g_n(t)\quad{a.e} \quad\text on \quad [0,T],\
\ \ \ \ v_n(0)=0, \
\end{equation*}
and 
\begin{equation*}
\dot{v}(t)+\A(t)v(t)=g(t) \quad{a.e} \quad\text on \quad [0,T],\
\ \ \ \ v(0)=0. \
\end{equation*}
By the uniqueness of solvability, $u_n=v_n+\vartheta$ and $u=v+\vartheta.$ Therefore, using the result of the first part of the proof we obtain
\begin{align*}
\|u_n-u\|_{MR_2(V,V')}=&\|v_n-v\|_{MR_2(V,V')}=\|(\A_n+\B)^{-1}g_n-(\A+\B)^{-1}g\|_{MR_2(V,V')}
\\&\leq \|(\A_n+\B)^{-1}(g_n-g)\|_{MR_2(V,V')}+\|(\A_n+\B)^{-1}g-(\A+\B)^{-1}g\|_{MR_2(V,V')}
\\&\leq \const\|\A_n\vartheta-\A\vartheta\|_{L^2(0,T;V')}+\const d_n\|g\|_{L^2(0,T;V')}
\\&\leq  \const d_n \Big[\|\vartheta\|_{L^2(0,T;V)}+\|\dot\vartheta\|_{L^2(0,T;V')}+\|f\|_{L^2(0,T;V')}\Big]
\\&\leq \const d_n \Big[\|u_0\|_{H}+\|f\|_{L^2(0,T;V')}\Big].
\end{align*}
 This completes the proof.
\end{proof}
Recall that the non-autonomous Cauchy problem (\ref{Abstract Cauchy problem 0}) is said to have $L^2$-maximal regularity in $H$ if for each $u_0\in V$ and $f\in L^2(0,T;H)$ the solution $u$ belongs to $MR_2(V,H):=L^2(0,T;V)\cap H^1(0,T;H).$ The next results is an easy consequence of Theorem \ref{proposition: uniform approx V'}.

\begin{corollary} Assume that the approximation problems \ref{Approximation Abstract Cauchy problem} has $L^2$-maximal regularity in $H.$ Let $u\in V$ and $f\in L^2(0,T;H)$ and let $(u_n)_{n\in\N}$ the sequence of solutions of (\ref{Approximation Abstract Cauchy problem}). If $(u_n)_{n\in\N}$ converges weakly in $MR_2(V,H),$ then the limit problem (\ref{Abstract Cauchy problem 0}) has also $L^2$-maximal regularity in $H$ and $u$ is equal to the weak limits of $(u_n)_{n\in\N}.$

\end{corollary}

\section{$L^2$-maximal regularity in $H:$ a weak approximation\label{weak approximation H}}
\noindent Let $\fra, \fra_n:[0,T]\times V\times V\lra \C$  be a closed non-autonomous forms  satisfying (\ref{eq:continuity-nonaut-introduction})  with the same constants $\beta, \alpha$ and $M>0.$ 
In this section we assume that there exist $0\leq \gamma< 1,$  a sequence of non-decreasing continuous  function  $\omega_n:[0,T]\longrightarrow [0,+\infty), n\in\N,$ and zero real sequence $(d_n)_{n\in\N}$ such that the following assumptions hold.
\begin{align*}
  &\mathbf{(  H_1)} \quad  |\fra_n(t,u,v)-\fra(t,u,v)| \le d_n \Vert u\Vert_{V} \Vert v\Vert_{V_\gamma}, \ t\in[0,T], \ u,v\in V,
 	\\  &\mathbf{(  H_2)} \quad  |\fra_n(t,u,v)-\fra_n(s,u,v)| \le\omega_n(|t-s|) \Vert u\Vert_{V} \Vert v\Vert_{V_\gamma}, \ t\in[0,T], \ u,v\in V,
\\ &\mathbf{(  H_3)} \quad  \displaystyle\sup_{t\in[0,T],n\in \N} \frac{\omega_n(t)}{t^{\gamma/2}}<\infty,  \text{ and }  \  \ \displaystyle\sup_{n\in \N}\int_0^T\frac{\omega_n(t)}{t^{1+\gamma/2}} {\rm d}t<\infty\ \
 \\ &\mathbf{(  H_4)} \quad \text{The approximation problem (\ref{Approximation Abstract Cauchy problem}) has $L^2$-maximal regularity in $H$ for every } n\in\N. \end{align*}
where $V_\gamma:=[H,V]_\gamma$ is the complex interpolation space. Note that \[V\hookrightarrow V_\gamma \hookrightarrow H\hookrightarrow V_\gamma'\hookrightarrow V'\] with continuous embeddings. Remark that condition $\bf(H_2)$ implies that $\A_n(t)-\A_n(s)\in \L(V,V_\gamma')$ and
\begin{equation}\label{Eq0: Dini condition operators}
\|\A_n(t)-\A_n(s)\|_{\L(V,V_\gamma')}\leq \omega_n(|t-s|), \quad  t,s\in[0,T], n\in\N.
\end{equation}

\par The following proposition is of great interest for this paper.

\begin{proposition}\label{lemma: estimations for general from}\cite[Section 2]{Ar-Mo15} Let $b$ be any sesquilinear form that satisfies (\ref{eq:continuity-nonaut-introduction}) with the same constants $M$ and $\alpha$ and let $\gamma\in [0,1[.$ Let $\B$ and $B$ be the associated operators  on $V'$ and $H,$ respectively. Then there exists a constant $c>0$ which depends only on $M,\alpha, \gamma$ and $c_H$ such that
	
	\begin{enumerate}
		\item \label{Eq1: estimation resolvent}$\displaystyle \|(\lambda-\B)^{-1}\|_{\L(V_{\gamma}',H)}\leq \frac{c}{(1+\mid\lambda\mid)^{1-\frac{\gamma}{2}}},$
		\item \label{Eq2: estimation resolvent}$\displaystyle \|(\lambda-\B)^{-1}\|_{\L(V)}\leq \frac{c}{1+\mid\lambda\mid},$
		\item \label{Eq3: estimation resolvent}$\displaystyle \|(\lambda-\B)^{-1}\|_{\L(H,V)}\leq \frac{c}{(1+\mid\lambda\mid)^{\frac{1}{2}}},$
		\item \label{Eq4: estimation resolvent} $\displaystyle \|(\lambda-\B)^{-1}\|_{\L(V',H)}\leq \frac{c}{(1+\mid\lambda\mid)^{\frac{1}{2}}},$
		\item\label{Eq6: estimation resolvent} $\displaystyle \|(\lambda-\B)^{-1}\|_{\L(V'_\gamma,V)}\leq \frac{c}{(1+\mid\lambda\mid)^{\frac{1-\gamma}{2}}},$
		\item \label{Eq2: estimation  semigroup}
		$\displaystyle\|e^{-s\B}\|_{\L(V_\gamma',H)}\leq\frac{c}{s^{\gamma/2}},$
		\item \label{Eq3: estimation  semigroup}
		$\displaystyle\|e^{-sB}\|_{\L(V_\gamma',V)}\leq\frac{c}{s^{\frac{1+\gamma}{2}}},$
		\item \label{Eq4: estimation  semigroup}
		$\displaystyle\|e^{-s\B}\|_{\L(V',V)}\leq\frac{c}{s^{\frac{1}{2}}},$
		\item \label{analytic estimation}
		$\displaystyle\|Be^{-sB}\|_{\L(H)}\leq \frac{c }{s},$
		\item \label{analytic estimation in V}
		$\displaystyle\|e^{-sB}\|_{\L(V)}\leq c$
	\end{enumerate}
		for each $t\in[0,T], s\geq 0$ and $\lambda\notin \Sigma_\theta:=\{re^{i\varphi}: r>0, |\varphi|<\theta\}.$
\end{proposition}

\begin{remark}\label{Remark: estimations for general from} All estimates in Proposition \ref{lemma: estimations for general from} holds for $\A_n(t)$ and $\A(t)$ with constant independent of $n$ and $t\in[0,T],$ since $\fra_n \fra$ satisfies (\ref{eq:continuity-nonaut-introduction}) with the same constants $M,\beta$ and $\alpha,$ also $\gamma$ and $c_H$  does not depend on $n$ and $t\in[0,T].$
\end{remark}
 \begin{notation} To keep notations simple as possible  we will in the sequel denote all positive constants depending on $M,\alpha, \gamma, c_H$ and $T$ that appear in proofs and theorems uniformly as $\const>0.$
\end{notation}
\medskip

For each $f\in L^2(0,T;H)$ and $u_0\in V,$  the solutions $u_n, n\in\N,$ of (\ref{Approximation Abstract Cauchy problem}) satisfies the following key formula
\begin{equation}\label{Formula for the solution}
u_n(t)=e^{-tA_n(t)}u_0+\int_{0}^te^{-(t-s)\A_n(t)}f(s){\rm  d}s+\int_{0}^te^{-(t-s)A_n(t)}(\A_n(t)-\A_n(s))u_n(s){\rm  d}s
\end{equation}
for all $t\in[0,T].$ This formula is due to Acquistapace and Terreni \cite{Ac-Ter87} and was proved in a more general setting in \cite[Proposition 3.5]{Ar-Mo15}. In the sequel we will use the following notations:

\begin{equation}\label{Formula for the solution 2}
u_{n,1}(t):=e^{-tA_n(t)}u_0, \qquad  u_{n,2}(t):=\int_{0}^t e^{-(t-s)A_n(t)}f(s){\rm  d}s.
\end{equation}
With this notation we can state the main result of this section which, in particular, shows that the limit problem  (\ref{Abstract Cauchy problem 0}) also has $L^2$-maximal regularity in $H.$ 
\begin{theorem}\label{main theorem:l2 max reg in H} Assume that the assumptions $\bf (H_1)$-$\bf (H_4)$ holds. Then the  problem  (\ref{Abstract Cauchy problem 0}) also has $L^2$-maximal regularity in $H.$ Moreover, if $f\in L^2(0,T;H)$ and $u_0\in V$ and  $(u_n)_{n\in\N}\subset MR_2(V,H)$ is the sequence of
	the unique solutions of (\ref{Approximation Abstract Cauchy problem}), then $(u_n)_{n\in\N}$ converges weakly in $MR_2(V,H)$ and $u:={\rm w}-\lim\limits_{n\to \infty}u_n$ satisfies (\ref{Abstract Cauchy problem 0}).
\end{theorem}
\noindent For the proof we need first some preliminary lemmas. Using the same argument as in the proof of \cite[Theorem 4.1]{Ar-Mo15}, the next two lemmas follow thanks to $\bf (H_1)$-$\bf (H_3)$ and Remark \ref{Remark: estimations for general from}. 

\begin{lemma}\label{Lemma: Invertibility of I-QLambda} Assume that the assumptions $\bf (H_1)$-$\bf (H_4)$ holds.
	Let $Q_n^\mu:L^2(0,T;H)\to L^2(0,T;H)$ denotes the linear operator  defined for all $g\in L^2(0,T;H)$ and $\mu\geq 0$ by
	\begin{equation}\label{def l operator Q}
	(Q_n^\mu g)(t):=\int_{0}^t(\A_n(t)+\mu)e^{-(t-s)(\A_n(t)+\mu)}(\A_n(t)-\A_n(s))(\A_n(s)+\mu)^{-1}g(s){\rm  d}s\ \quad t\textrm{-a.e}.
	\end{equation}
	Then $\lim\limits_{\mu \to \infty}\|Q_n^\mu\|_{L^2(0,T;H)}=0$ uniformly on $n$ and thus $I-Q_n^\mu$ is invertible on $L^2(0,T;H)$ for $\mu$ large enough and for all $n.$
\end{lemma}
\begin{lemma} \label{Lemma: The uniform  Boundedness in $L^2(0,T;H)$} Assume that the assumptions $\bf (H_1)$-$\bf (H_3)$ holds. The following tow estimates 
	\begin{align*}
	\|\A_n u_{n,1}\|_{L^2(0,T;H)}&\leq  \const \|u_0\|_{V}^2,
	\\
	\|\A_n u_{n,2}\|_{L^2(0,T;H)}&\leq \const\|f\|_{L^2(0,T;H)}
	\end{align*}
	hold.
\end{lemma}

\par Now we can give the proof of Theorem \ref{main theorem:l2 max reg in H}.

%
\begin{proof}(\textit{of Theorem} \ref{main theorem:l2 max reg in H}) According to Lemma \ref{Lemma: Invertibility of I-QLambda} and replacing $\A_n(t)$ with $\A_m(t)_n+\mu,$ we may assume  without loss of generality that $Q_n=Q^\mu_n$ satisfies $\|Q_n\|_{\L(L^2(0,T;H))}< 1,$ and then $I-Q_n$ is invertible by the Neumann series. We deduce from (\ref{Formula for the solution}) that
	\[\dot u_n=\A_n u_n=(I-Q_n)^{-1}(\A_n u_n^1+\A_n u_n^2).\]
	This equality and Lemma \ref{Lemma: The uniform  Boundedness in $L^2(0,T;H)$},  yield the estimate
	\begin{equation}
	\Vert \dot u_n\Vert_{L^2(0,T;H)}\leq {\bf c}\big[\Vert u_0\Vert_V+\Vert f\Vert_{L^2(0,T;H)}\big]
	\end{equation}
 Since for all $t\in [0,T]$ one has $u_n(t)=u_n(0)+\int_0^t \dot u_n(s){\rm d}s,$ we conclude that
	\begin{equation}
	\Vert u_n\Vert_{H^1(0,T;H)}\leq {\bf c}\big[\Vert u_0\Vert_V+\Vert f\Vert_{L^2(0,T;H)}\big].
	\end{equation}
	Then there exists a subsequence of $(u_n),$ still denoted by $(u_n)$ that converges weakly to some $v\in H^1(0,T;H)$
	\par\noindent On the other hand, the Cauchy problem (\ref{Abstract Cauchy problem 0}) has a unique solution $u\in MR_2(V,V'),$ and $(u_n)$ converges strongly to $u$ on $MR_2(V,V')$  by Theorem \ref{proposition: uniform approx V'}. We conclude by uniqueness of limits that $u= v\in H^1(0,T;H).$ This completes the proof.
\end{proof}
\section{$L^2$-maximal regularity in $H:$ uniform approximation \label{Uniform Convergence}}
\noindent Assume that  $\fra$ and $\fra_n$  are as in Section \ref{weak approximation H}. Let $(f,u_0)\in L^2(0,T;H)\times V$ and let $u, u_n\in MR_2(V,H)$ be the solutions of (\ref{Abstract Cauchy problem 0}) and (\ref{Approximation Abstract Cauchy problem}), respectively. In the previous section we have seen  that $(u_n)_{n\in\N}$ converges weakly to $u$ with respect to the norm of $MR_2(V,H).$ The aim of this
section is to prove that this convergence holds for the strong topology of $MR_2(V,H)$ and uniformly on the initial data $u_0$ and $f.$ To this end, we impose the following additional conditions:
\begin{align*}
 & {\bf (H_5)}  \lim\limits_{n\to \infty}d_nn^{\gamma/2}=0 \text{ and the sequence } (d_n)_{n\in\N} \text{ is decreasing}.\qquad \qquad \qquad \qquad\qquad\qquad\qquad
\\ & {\bf (H_6)}  \lim\limits_{n\to \infty} \displaystyle\int_{0}^{1/n}\frac{\omega_n(r)}{r^{{1+\gamma/2}}}  {\rm d}r= 0. 
\end{align*}

Recall that $-A_n(t)$ generates a holomorphic $C_0$-semigroup (of angle $\theta:=\frac{\pi}{2}-\arctan(\frac{M}{\alpha})$) $e^{-s A_n(t)}$ on $H$ which is the restriction to $H$ of $e^{-\cdot A_n(t)},$ and we have
\begin{equation}\label{analytic representation}
e^{-\cdot A_n(t)}=\frac{1}{2i\pi}\int_\Gamma e^{\cdot\mu} (\mu+A_n(t))^{-1}{\rm d}\mu
\end{equation}
where $\Gamma:=\{re^{\pm \varphi}:\ r>0\}$ for some fixed $\varphi\in (\theta,\frac{\pi}{2}).$

\begin{theorem}\label{convergence uniform} Assume that the assumptions $\bf (H_1)$-$\bf (H_6)$ holds. Then there exists a positive constant $\textbf{c}>0$ depending only on $M,\alpha,\gamma$ and $c_H$  such that 
	\begin{equation}\label{estimation of the error of the convergence2}
	\|\dot u-\dot u_n\|_{L^2(0,T;H)}\leq \textbf{c}\Big[(1+n^{\gamma/2})d_n+\int_{0}^{1/n}\frac{\omega_n(r)}{r^{{1+\gamma/2}}}  {\rm d}r\Big]
\Big[\|f\|_{L^2(0,T;H)}+\|u_0\|_V\Big].
	\end{equation}
Thus $(u_n)_{n\in\N}$ convergences to $u$ for the strong topology of $MR_2(V,H)$ and uniformly on the initial data $u_0$ and $f.$	
	
\end{theorem}
\begin{proof} We only have to prove (\ref{estimation of the error of the convergence2})  the uniform convergence with respect to $u_0,f$ in $MR_2(V,H)$ becomes obvious. Indeed, we known from Theorem \ref{proposition: uniform approx V'} that $u_n \longrightarrow u$ in $L^2(0,T;V)$ uniformly on the initial data $u_0$ and the homogeneity $f.$
\par\noindent We will use the representation formula (\ref{Formula for the solution}) and (\ref{Formula for the solution 2}). We proceed by several steps. Let $m,k\in\N$ and set $n:=m+k$ and $d_{n,m}:=d_n+d_m.$
	\par $(a)$ First, we estimate  $\A_n u_{n,1}-\A_m u_{m,1}$  in $L^2(0,T;H).$ Let $t\ne 0.$ Using $\bf (H_1)$ we obtain  and the estimates (\ref{Eq2: estimation  semigroup}) and  (\ref{analytic estimation in V}) in Proposition \ref{lemma: estimations for general from} that
	\begin{align}
	\nonumber\|\A_n(t) u_{n,1}(t)-&\A_m(t)u_{m,1}(t)\|_H=\|\A_n(t)e^{-t \A_n(t)}u_0-\A_m(t)e^{-t \A_m(t)}u_0\|_H
	\\\nonumber&\leq  \|e^{-t \A_n(t)}[\A_n(t)u_0-\A_m(t)u_0]\|_H+
	\|[e^{-t \A_n(t)}-e^{-t \A_m(t)}]\A_n(t)u_0\|_H
	\\\nonumber&=\|e^{-t \A_n(t)}[\A_n(t)u_0-\A_m(t)u_0]\|_H+\int_0^t \|e^{-(t-s) \A_n(t)}(\A_n(t)-\A_m(t))e^{-s \A_m(t)}u_0\|_H
	\\&\label{eq1: proof of Theorem on convergence uniform}  \leq \const d_{n,m}\left(\frac{1}{t^{\gamma/2}}+\int_0^t\frac{1}{s^{\gamma/2}}\rm{ ds}\right)\|u_0\|_{V}.
	\end{align}
	Similarly, combining the estimates (\ref{Eq1: estimation resolvent}) and (\ref{Eq3: estimation resolvent}) in Proposition \ref{lemma: estimations for general from} and  the estimate (\ref{Eq2:Dini condition operators}) in Proposition \ref{lemma: estimations for general from} we obtain
	\begin{align}\nonumber
	\|\A_n(t) &u_{n,2}(t)-\A_m(t)u_{2}(t)\|_H
	\\\nonumber\leq\int_0^t& \|[\A_n(t)e^{-(t-s)A_n(t)}-\A_m(t)e^{-(t-s)A(t)}]f(s)\|_Hds
	\\\nonumber&\leq\frac{1}{2\pi}\int_{0}^t\int_{\Gamma}\mid\lambda\mid e^{-(t-s)\Re\lambda}\|(\lambda-\A_n(t))^{-1}(\A_n(t)-\A_m(t))(\lambda-\A_m(t))^{-1}f(s)\|_{H}{\rm  d}\lambda{\rm  d}s
	\\\nonumber&\leq \const\int_{0}^td_{n,m}\int_{\Gamma}
	\frac{e^{-(t-s)\Re\lambda}}{\mid\lambda\mid^{\frac{1-\lambda}{2}}}\|f(s)\|_{H}{\rm  d}\lambda{\rm  d}s
	\\\nonumber&=\const d_{n,m}\int_{0}^t \|f(s)\|_{H}\int_{0}^\infty
	\frac{e^{-(t-s)r\cos(\nu)}}{r^{\frac{1-\gamma}{2}}}{\rm  d}r{\rm  d}s
	\\	\nonumber&=\const d_{n,m}\int_{0}^t \|f(s)\|_{H}\int_{0}^\infty\frac{e^{-\rho\cos(\theta)}}{(\frac{\rho}{t-s})^{-\frac{1+\gamma}{2}}}{{\rm  d}\rho}{\rm  d}s
	\\&=\label{eq2: proof of Theorem on convergence uniform}
\const d_{n,m}\int_{0}^\infty
	\frac{e^{-\rho\cos(\nu)}}{\rho^{\frac{1-\gamma}{2}}}{\rm  d\rho}\int_{0}^t\|f(s)\|_{H}(t-s)^{-\frac{1+\gamma}{2}}{\rm  d}s.
	\end{align}
	The last integral is well defined since the function $h:\R\to \R$ given by $h(t)=t^{-\frac{1+\gamma}{2}}$ for $t\in]0,T]$ and $h(t)=0$ for $t\in ]-\infty,0]\cap]T,+\infty[$  belongs to $L^1(\R)$ because $\frac{1+\gamma}{2}<1.$ The estimates (\ref{eq1: proof of Theorem on convergence uniform}) and (\ref{eq2: proof of Theorem on convergence uniform}) yield, respectively,
	\begin{equation*}
	\|\A_n u_{n,1}-\A_m u_{m,1}\|_{L^2(0,T;H)}\leq\const d_{n,m}\|u_0\|_{V}
	\end{equation*}
	and
	\begin{equation*}\label{estimation error Au2-Au2lambda}
	\|\A_n u_{n,2}-\A_m u_{m,2}\|_{L^2(0,T;H)}\leq \const d_{n,m}\|f\|_{L(0,T;H)}.
	\end{equation*}

	\par\noindent $(b)$ Next, we  prove  the following estimate
	\begin{equation}\label{estimation of the error of the convergence for Q}
	\|Q_n -Q_m\|_{\L(L^2(0,T;H))}\leq \const\left[ d_{n,m}+m^{\gamma/2}d_m+ n^{\gamma/2}d_n+\int_{0}^{1/n}\frac{\omega_n(r)}{r^{{1+\gamma/2}}}  {\rm d}r+\int_{0}^{1/m}\frac{\omega_m(r)}{r^{{1+\gamma/2}}}  {\rm d}r\right]
	\end{equation}
	where $Q_m:L^2(0,T;H)\longrightarrow L^2(0,T;H)$ is defined  (\ref{def l operator Q}). To this end,  for  $g\in L^2(0,T;H)$ and $t\in [0,T]$ we write 
	
	\begin{align*}\label{eq1: proof strong conv Qlambda}
	\|(Q_n g)(t)&-(Q_m g)(t)\|_H\\&\leq \int_0^t\|\A_n(t)e^{-(t-s)\A_n(t)}(\A_n(t)-\A_n(s))(A^{-1}_n(s)-A^{-1}_m(s))g(s)\|_H ds
	\\&\quad + \int_0^t\|\A_n(t)e^{-(t-s)\A_n(t)}(\A_n(t)-\A_m(t)-\A_n(s)+\A_m(s))A^{-1}_m(s)g(s)\|_H ds
	\\&\quad + \int_0^t\|\big(\A_n(t)e^{-(t-s)\A_n(t)}-\A_m(t)e^{-(t-s)A(t)}\big)\big(\A_m(t)-\A_m(s)\big)A^{-1}_m(s)g(s)\|_H ds
	\\&=I_{n,m,1}(t)+I_{n,m,2}(t)+I_{n,m,3}(t)
	\end{align*}
	Replacing $\A_m(s)$ by $\A_m(s)+\mu$ and according to Proposition \ref{lemma: estimations for general from} we may assume $\|\A^{-1}_n(s)\|_{\L(V_\gamma',V)}\leq \const$ and $\|\A^{-1}_n(s)\|_{\L(H,V)}\leq \const.$ Next, by the estimates (\ref{Eq2: estimation  semigroup}) and (\ref{analytic estimation}) in Proposition \ref{lemma: estimations for general from}  together with $\bf (H_1)-\bf(H_2)$, we obtain
	\begin{align*}
	I_{n,m,1}(t)=&\int_0^t\|\A_n(t)e^{-\frac{t-s}{2}\A_n(t)}e^{-\frac{t-s}{2}\A_n(t)}
	(\A_n(t)-\A_n(s))(\A^{-1}_n(s)-\A^{-1}_m(s))g(s)\|_H ds
	\\&\leq \const\int_{0}^t \frac{\omega_n(t-s)}{(t-s)^{1+\gamma/2}}\|(\A^{-1}_n(s)-\A^{-1}_m(s))g(s)\|_Vds
	\\&=\const\int_0^t\frac{\omega_n(t-s)}{(t-s)^{1+\gamma/2}}\|(\A^{-1}_n(s)(\A_n(s)-\A_m(s))\A^{-1}(s))g(s)\|_Vds
	\\&\leq \const d_{n,m}\int_0^t \frac{\omega_n(t-s)}{(t-s)^{1+\gamma/2}}\|g(s)\|_{H}ds
	\\&=\const d_{n,m} (h_{n}\ast\|g(\cdot)\|_{H})(t),
	\end{align*}
	where $h_{n}(t):=\omega_n(t)t^{-1-\gamma/2}$ for $t\in[0,T]$ and $h_{n}(t):=0$ for $t\in (-\infty,0[\cap]T,+\infty).$  The assumption $\bf (H_3)$ implies  that  $h_n\in L^1(\R)$ and that $(\| h_n\|_{L^1(\R)})_{n\in\N}$ is bounded. Therefore, we obtain
	\begin{equation}\label{estimate proof}
	\int_0^T I_{n,m,1}^2(s){\rm d}s\leq \const d_{n,m}^2\int_0^T\|g(s)\|_{H}^2ds.
	\end{equation}

	\par\noindent Again by estimates (\ref{Eq2: estimation  semigroup}) and (\ref{analytic estimation}) in Proposition \ref{lemma: estimations for general from}, we obtain for the second term $I_{n,m,2}$
	\begin{align*}
	&I_{n,m,2}(t):=\int_0^t\|\A_n(t)e^{-(t-s)\A_n(t)}(\A_n(t)-\A_m(t)-\A_n(s)+\A_m(s))A^{-1}(s)g(s)\|_H ds
	\\&\leq \const \int_{0}^t \|(\A_n(t)-\A_m(t)-\A_n(s)+\A_m(s))\|_{\L(V_\gamma',H)}
	\frac{\|g(s)\|_{H}}{(t-s)^{1+\gamma/2}}{\rm  d}s
	\\&\leq\const\int_{0}^t \kappa_{n,m}(t-s)\|g(s)\|_{H}{\rm  d}s
	\end{align*}
	where
	\begin{equation*}\label{Kappa}\kappa_{n,m}
	(t):=\left\{%
	\begin{array}{ll}
	[\omega_m(t)+\omega_n(t)]t^{-(1+\frac{\gamma}{2})} & \hbox{ if } \  0\leq  t<\frac{1}{n},\\
	4d_{n,m}t^{-(1+\frac{\gamma}{2})}& \hbox{ if } \  \frac{1}{n}< t\leq T, \\
	0& \hbox{ if } \   t\in ]-\infty,0]\cap]T,+\infty[. \\
	\end{array}%
	\right. \end{equation*}
 Thanks to $\bf (H_3)$ and $\bf (H_5)$,
	$t\mapsto \kappa_{n,m}(t)$ belongs to $L^1(\R)$, and by a simple calculation we obtain
	
	\begin{equation}\label{estimation kappa}\|\kappa_{n,m}\|_{L^1(\R)}\leq \const\left[ m^{\gamma/2}d_m+ n^{\gamma/2}d_n+\int_{0}^{1/n}\frac{\omega_n(r)}{r^{{1+\gamma/2}}}  {\rm d}r+\int_{0}^{1/m}\frac{\omega_m(r)}{r^{{1+\gamma/2}}}  {\rm d}r\right].\end{equation}
Therefore,
	\begin{equation}\label{I2}
	\int_0^T I_{n,,m,2}^2(s){\rm d}s\leq  {\bf c}\left[  m^{\gamma/2}d_m+ n^{\gamma/2}d_n+\int_{0}^{1/n}\frac{\omega_n(r)}{r^{{1+\gamma/2}}}  {\rm d}r+\int_{0}^{1/m}\frac{\omega_m(r)}{r^{{1+\gamma/2}}}  {\rm d}r\right]^2\int_0^T\|g(s)\|_{H}^2ds.
	\end{equation}
	\par\noindent For the last term $I_{n,m,3}(t),$ we set $\tilde{g}_m(t,\cdot):=(\A_m(t)-\A_m(\cdot))A_m^{-1} (\cdot)g(\cdot).$ Again by assumptions $\bf (H_1)$ and  (\ref{Eq4: estimation resolvent}) and (\ref{Eq6: estimation resolvent}) from Proposition \ref{lemma: estimations for general from} and we obtain
	\begin{align*}
	&I_{n,m,3}(t):=\int_0^t\|(\A_n(t)e^{-(t-s)\A_n(t)}-\A_m(t)e^{-(t-s)A_\Gamma(t)})\tilde{g}_m(t,s)\|_H ds
	\\&\leq\frac{1}{2\pi}\int_{0}^t\int_{\Gamma}\mid\lambda\mid e^{-(t-s)\Re\lambda}
	\|(\lambda-\A_n(t))^{-1}(\A_n(t)-\A_m(t))(\lambda-\A_m(t))^{-1}\tilde{g}_m(t,s)\|_{H}{\rm  d}\lambda{\rm  d}s
	\\&\leq \frac{1}{2\pi}\int_{0}^t\int_{\Gamma}\mid\lambda\mid e^{-(t-s)\Re\lambda}
	\|(\lambda-\A_n(t))^{-1}\|_{\L(V',H)}\|(\A_n(t)-\A_m(t))(\lambda-\A_m(t))^{-1}\tilde{g}_m(t,s)\|_{V'}{\rm  d}\lambda{\rm  d}s
	\\&	\leq \frac{1}{2\pi}\int_{0}^t\int_{\Gamma}\mid\lambda\mid e^{-(t-s)\Re\lambda}
	\|(\lambda-\A_n(t))^{-1}\|_{\L(V',H)}\|(\A_n(t)-\A_m(t))(\lambda-\A_m(t))^{-1}\tilde{g}_m(t,s)\|_{V'_\gamma}{\rm  d}\lambda{\rm  d}s
	\\&\leq \const\int_{0}^t\int_{\Gamma}\mid\lambda\mid e^{-(t-s)\Re\lambda}
	\frac{d_{n,m}}{(1+|\lambda|)^{1/2}}\|(\lambda-\A_m(t))^{-1}\|_{\L(V_\gamma',V)}\|\tilde{g}_m(t,s)\|_{V_\gamma'}{\rm  d}\lambda{\rm  d}s
	\\&\leq \const d_{n,m}\int_{0}^t\int_{\Gamma}
	\frac{\mid\lambda\mid e^{-(t-s)\Re\lambda}}{(1+|\lambda|)^{1-\frac{\gamma}{2}}}\|\tilde{g}_m(t,s)\|_{V_\gamma'}{\rm  d}\lambda{\rm  d}s
	\\&\leq \const d_{n,m}\int_{0}^t\int_{0}^\infty
	r^{\frac{\gamma}{2}} e^{-(t-s)r\cos(\nu)}\|\tilde{g}_m(t,s)\|_{V_\gamma'}{\rm  d}r{\rm  d}s
	\end{align*}
 Next, since
	\[\|\tilde{g}_m(t,s)\|_{V_\gamma'}\leq \omega_m(|t-s|)\|A^{-1}_m(t)\|_{\L(H,V)}\|g(s)\|_H,\] 
it follows
	
	\begin{align*}
	I_{n,m,3}(t)&\leq \const d_{n,m}\int_{0}^\infty\frac{e^{-\rho\cos(\nu)}}{\rho^{-\gamma/2}}{\rm  d}\rho \int_{0}^t\frac{\omega_n(t-s)}{(t-s)^{1+\frac{\gamma}{2}}}\|g(s)\|_{H}{\rm  d}s
	\\&=\const d_{n,m}(h\ast\|g(\cdot)\|_{H})(t)\int_{0}^\infty\frac{e^{-\rho\cos(\nu)}}{\rho^{-\gamma/2}}{\rm  d}\rho 
	\end{align*}
	Using the same argument as that used above for (\ref{estimate proof}) one obtain 
	\begin{equation}
	\int_0^T I_{n,m,3}^2(s){\rm d}s\leq \const d_{n,m}^2\int_0^T\|g(s)\|_{H}^2ds.
	\end{equation}
This together with \ref{I2} and (\ref{estimate proof}) give the desired estimate (\ref{estimation of the error of the convergence for Q}).

	\par\noindent $(c)$ Using  Lemma \ref{Lemma: The uniform  Boundedness in $L^2(0,T;H)$} we conclude from $a)-b)$ that
	\begin{align*}
	\nonumber &\|\A_n u_n-\A_m u_m\|_{L^2(0,T;H)}
	\\&\leq \|(I-Q_n)^{-1}(\A_n u_{n,1}-\A u_{m,1})\|_{L^2(0,T;H)}
	+\|(I-Q_n)^{-1}[\A_n u_{n,2}-\A u_{m,2}]\|_{L^2(0,T;H)}
	\\\nonumber&\qquad+\|(I-Q_n)^{-1}(Q_m-Q_n)(I-Q_m)^{-1}(\A u_{m,1}+\A u_{m,2})\|_{L^2(0,T;H)}
	\\&\label{estimation final 1}\leq \const\left[d_{n,m}+n^{\gamma/2}d_n+m^{\gamma/2}d_m+  \int_{0}^{1/n}\frac{\omega_n(r)}{r^{{1+\gamma/2}}}  {\rm d}r+\int_{0}^{1/m}\frac{\omega_m(r)}{r^{{1+\gamma/2}}}  {\rm d}r\right]\Big[\|u_0\|_{V}+\|f\|_{L^2(0,T;H)}\Big].
	\end{align*}
Finally, since $(u_n)_{n\in\N}$ satisfies (\ref{Approximation Abstract Cauchy problem}),  we conclude that 

	\begin{equation}\label{estimation final 2}
	\|\dot u_{k+m}-\dot u_m\|_{L^2(0,T;H)} \leq \const_{k,m}\Big[\|u_0\|_{V}+\|f\|_{L^2(0,T;H)}\Big]
	\end{equation} 
with
	
\[\const_{n,m}:=\const\left[d_{k+m}+d_m+(k+m)^{\gamma/2}d_{k+m}+m^{\gamma/2}d_m+  \int_{0}^{\frac{1}{k+m}}\frac{\omega_{k+m}(r)}{r^{{1+\gamma/2}}}  {\rm d}r+\int_{0}^{\frac{1}{m}}\frac{\omega_m(r)}{r^{{1+\gamma/2}}}  {\rm d}r\right].
\]
Thus $(\dot u_n)$ is a Cauchy sequence. The limits coincides with $u$ according to Corollary (\ref{proposition: uniform approx V'}). This completes the proof of (\ref{estimation of the error of the convergence2}) by taking $k\lra\infty$ in (\ref{estimation final 2}).

\end{proof}
%
\section{Uniform approximation on $C(0,T;V)$}

Let $\fra, \fra_n:[0,T]\times V\times V\longrightarrow \C$ are as in Section \ref{Uniform Convergence}. Additionally, we assume that $(u_n)_{n}\subset C([0,T],V).$ Then we show in this section that $(u_n)_{n}$ converges in $C([0,T],V)$ uniformly on $(f,u_0).$ 

\begin{proposition}\label{continuity of the solution} With the notations of Section \ref{Uniform Convergence} the following estimate holds
	\begin{equation}\label{estimation of the error of the convergence3}
	\|u_{m+k}-u_m\|_{C(0,T;V)}\leq \const_{k,m}\Big[\|u_0\|_V+\|f\|_{L^2(0,T;H)}\Big]
	\end{equation}
	for all $k,m\in \N.$
\end{proposition}

In view of Theorem (\ref{proposition: uniform approx V'}), the following is then true and follows immediately from Proposition (\ref{continuity of the solution}).
\begin{corollary}\label{corollary continuity of solution} Let $(f,u_0)\in L^2(0,T;H)\times V$ and let $u\in MR_2(V,H)$ be the solution of (\ref{Abstract Cauchy problem 0}). Then $u\in C(0,T;V)$ and 
		\begin{equation}\label{estimation of the error of the convergence4}
		\|u_{n}-u\|_{C(0,T;V)}\leq \textbf{c}\left[d_{n}+n^{\gamma/2}d_n\right]\Big[\|u_0\|_V+\|f\|_{L^2(0,T;H)}\Big]
		\end{equation}
holds.
\end{corollary}

\begin{proof} We will proceed similarly as in the proof of Theorem
	\ref{convergence uniform}. Let $m,k\in\N$ and set $n:=m+k$ and $d_{n,m}:=d_n+d_m.$
	\par\noindent \textit{Step a:}  By using  (\ref{Eq2: estimation resolvent}) and (\ref{Eq6: estimation resolvent}) in Proposition \ref{lemma: estimations for general from} for $(\lambda-\A_n(t))^{-1}$ and $(\lambda-\A_m(t))^{-1},$ respectively, and $\bf (H_1)$ we obtain  for every $t\in[0,T]$ that
	\begin{align*}
	\|u_{1,n}(t)&-u_{1,m}(t)\|_V\leq \frac{1}{2\pi}\int_\Gamma e^{-t\Re \lambda}\|(\lambda-\A_n(t))^{-1}(\A_n(t)-\A_m (t))(\lambda-\A_m(t))^{-1}u_0\|_V d\lambda
	\\&\leq\const d_{n,m}\int_\Gamma \frac{e^{-t\Re \lambda}}{(1+|\lambda|)^\frac{3-\gamma}{2}}d\lambda \|u_0\|_V
	\\&\leq \const d_{n,m}\Big(\int_0^\infty \frac{1}{(1+r)^\frac{3-\gamma}{2}}d r\Big) \|u_0\|_V.
 	\end{align*}
	
	\par\noindent \textit{Step b:} Again  the estimates  (\ref{Eq4: estimation resolvent}) and  (\ref{Eq6: estimation resolvent}) in Proposition \ref{lemma: estimations for general from} and formula $\bf (H_1)$ imply that
	\[\|(\lambda-\A_n(t))^{-1}(\A_n(t)-\A_m(t))(\lambda-\A_m(t))^{-1}f(s)\|_V\leq  \frac{\const d_{n,m}}{(1+|\lambda|)^{1-\frac{\gamma}{2}}}\|f(s)\|_H.\]
	Therefore, we obtain by using  Fubini's theorem that
	for all $\lambda\in \Gamma\setminus\{0\}$
	\begin{align*}
	\|u_{2,n}(t)&-u_{2,m}(t)\|_V=\|\int_0^t\frac{1}{2i\pi}\int_\Gamma e^{-(t-s) \lambda}(\lambda-\A_n(t))^{-1}(\A_n(t)-\A_m(t))(\lambda-\A_m(t))^{-1}f(s) d\lambda ds\|_V
	\\&\leq \const d_{n,m}\int_0^\infty \frac{1}{(1+r)^{1-\frac{\gamma}{2}}}\Big(\int_0^t e^{-(t-s)r\cos(\nu)}\|f(s)\|_H ds\Big) dr
	\\&\leq \const d_{n,m}\|f\|_{L^2(0,T;H)}\int_0^\infty \frac{1}{(1+r)^{1-\frac{\gamma}{2}}}\Big(\int_0^t e^{-2(t-s)r\cos(\nu)}ds\Big)^{1/2}  dr
	\\&\leq \const\frac{d_{n,m}}{\sqrt{2\cos(\nu)}}\Big(\int_0^\infty \frac{1}{\sqrt{r}(1+r)^{1-\frac{\gamma}{2}}}dr\Big)\|f\|_{L^2(0,T;H)}\leq \const d_{n,m}\|f\|_{L^2(0,T;H)}.
	\end{align*}
	
	\par\noindent\textit{Step c:} For each $h\in C(0,T;V)$ we set
	\[(P_n h)(t):=\int_{0}^te^{-(t-s)A_n(t)}(\A_n(t)-\A_n(s))h(s){\rm  d}s.\]
	From \cite[Lemma 4.5]{Ar-Mo15} we have $(P_n h)_{n\in\N}\subset C(0,T;V).$ Thanks to Proposition  \ref{Prop: Dini condition for Linear-approximation}
	and assumptions ${\bf (H_2)}$-${\bf (H_3)}$ one can prove in a similar way as in Step 3 of the proof of \cite[Theorem 4.4]{Ar-Mo15} (see also Step 3 of the proof of Lemma \ref{Lemma: Invertibility of I-QLambda}) that $\|P_n\|_{\L(C(0,T;V)}\leq 1/2$ and thus $I-P_n$ is invertible on $\L(C(0,T,V)).$ Therefore, we obtain by using  the representation formula (\ref{Formula for the solution})
	\begin{align} \label{Eq 1:Theorem: proof continuity}
	u_{n}-u_{m}=(I-P_n)^{-1}&(u_{n,1}-u_{m,1})+(I-P_m)^{-1}(u_{n,2}-u_{m,2})\\&\qquad +
	(I-P_n)^{-1}(P_n-P_m)(I-P_m)^{-1}(u_{m,1}+u_{m,2})
	\end{align}
	\par\noindent The term on the right hand side of (\ref{Eq 1:Theorem: proof continuity}) is treated in \textit{Step} a)-b). We need only to  estimate the difference $P_n-P_m$ on $\L(C(0,T;V)).$  For each $h\in C(0,T;V)$ and $t\in [0,T]$ we have
	\begin{align*}
	(&P_n h -P_m h)(t)\\&=\int_{0}^te^{-(t-s)A_n(t)}\big[\A_n(t)-\A_n(s)-\A_m(t)+\A_m(s)\big]h(s){\rm  d}s
	\\& \qquad\qquad +\int_{0}^t\big[e^{-(t-s)A_n(t)}-e^{-(t-s)A_m(t)}\big](\A_m(t)-\A_m(s))h(s){\rm  d}s.
	\\&=\int_{0}^te^{-(t-s)A_n(t)}\big[\A_n(t)-\A_n(s)-\A_m(t)+\A_m(s)\big]h(s){\rm  d}s
	\\&+\int_{0}^t\frac{1}{2i\pi}\int_\Gamma e^{-(t-s)\lambda}(\lambda-\A_n)^{-1}(A_n(t)-A_m(t))(\lambda-\A_m)^{-1}(\A_m(t)-\A_m(s))h(s){\rm  d}s.
	\end{align*}
	From $\bf (H_1)$-$\bf (H_2)$ and the estimate  (\ref{Eq6: estimation resolvent}) in Proposition \ref{lemma: estimations for general from} we have 
	\begin{align*}
	\|(\lambda-\A_n)^{-1}(A_n(t)-&A_m(t))(\lambda-\A_m)^{-1}(\A_m(t)-\A_m(s))h(s)\|_V
	\\&\leq \const d_{n,m}\frac{\omega_m(t-s)}{(1+|\lambda|)^{1-\gamma}}\|h(s)\|_V.
	\end{align*}
	Thus using $\bf (H_3)$, it follows 
	\begin{align*}
	\|\int_{0}^t&\frac{1}{2i\pi}\int_\Gamma \big[e^{-(t-s)\lambda}(\lambda-\A_n)^{-1}(A_n(t)-A_m(t))(\lambda-\A_m)^{-1}\big](\A_m(t)-\A_m(s))h(s){\rm  d}s \|_V
	\\&\leq \const d_{n,m}\int_{0}^t\frac{1}{\pi}\int_0^\infty e^{-(t-s)r\cos(\theta)}\frac{\omega_m(t-s)}{r^{1-\gamma}}\|h(s)\|_V dr ds
	\\&= \const d_{n,m}\int_{0}^t\frac{1}{\pi}\int_0^\infty e^{-\rho\cos(\theta)}\frac{\omega_m(t-s)}{\rho^{1-\gamma}}(t-s)^{1-\gamma}\|h(s)\|_V \frac{d\rho}{t-s} ds
	\\&\leq \const d_{n,m}\int_{0}^t\Big(\int_0^\infty \frac{e^{-\rho\cos(\nu)}}{\rho^{1-\gamma}}d\rho\Big)\frac{\omega_m(t-s)}{(t-s)^{\gamma}}\|h(s)\|_V d\rho ds
	\\& \leq \const d_{n,m} \|h\|_{C(0,T;V)}.
	\end{align*}
	Next, writing
	\begin{align*}e^{-(t-s)A_n(t)}&\big[\A_n(t)-\A_n(s)-\A_m(t)+\A_m(s)\big]h(s)
	\\&=A^{-1/2}_n(t)A^{1/2}_n(t)e^{-(t-s)A_n(t)}\big[\A_n(t)-\A_n(s)-\A_m(t)+\A_m(s)\big]
	\end{align*}
	then from (\ref{Eq3: estimation  semigroup}) and (\ref{Eq4: estimation  semigroup}) in Proposition \ref{lemma: estimations for general from} and the fact that $e^{-\cdot A_n(t)}$ is an analytic $C_0$-semigroup on $V$ we obtain 
	\begin{align*}
	&\int_{0}^t\|e^{-(t-s)A_n(t)}\big[\A_n(t)-\A_n(s)-\A_m(t)+\A_m(s)\big]h(s)\|_V{\rm  d}s
	\\ &\leq  \const \int_{0}^t \kappa_{n,m}(s)\|h(t-s)\|_V{\rm  d}s
	\end{align*}
	where $\kappa_{n,m}$ is defined in the proof of Theorem \ref{convergence uniform}. Therefore, using (\ref
	{estimation kappa}) we conclude 
	\begin{align*}\|u_{m+k}(t)-u_{m}(t)\|_V&\leq \const_{n,m}\Big[\|u_0\|_V+\|f\|_{L^2(0,T;H)}\big]
	\end{align*}
This complete the proof of the proposition.
	
\end{proof}

\section{Example: an affine approximation}
The aim of this section is to provide an explicit approximation of $\fra$ that satisfies the required hypothesis $\bf (H_1)-\bf(H_6).$ Recall that  $V,H$ denote two separable complex Hilbert spaces and $\fra:[0,T]\times V\times V\to \C$ is a non-autonomous closed form satisfying  (\ref{eq:continuity-nonaut-introduction}). Assume moreover that there exists $0\leq \gamma< 1$ and a non-decreasing continuous function $\omega:[0,T]\longrightarrow [0,+\infty)$ with
\begin{equation}\label{eq 1:Dini-condition} \sup_{t\in[0,T]} \frac{\omega(t)}{t^{\gamma/2}}<\infty, \end{equation}
\begin{equation}\label{eq 2:Dini-condition}
\int_0^T\frac{\omega(t)}{t^{1+\gamma/2}} {\rm d}t<\infty
\end{equation}
and
\begin{equation}\label{eq 3:Dini-condition}
|\fra(t,u,v)-\fra(s,u,v)| \le\omega(|t-s|) \Vert u\Vert_{V} \Vert v\Vert_{V_\gamma},  \quad t,s\in [0,T], u,v\in V.
\end{equation}
\begin{remark} We note that, the main example of a continuity modulus $\omega$ introduced above is the function $\omega(t)=t^{\eta}$ where $\eta>\gamma/2.$ This main example moreover satisfies 
\begin{equation}\label{model contiuity additional}
\lim\limits_{t\to 0}\frac{\omega(t)}{t^{\gamma/2}}=0.
\end{equation}
For general case, thanks  to  (\ref{eq 2:Dini-condition}), one can always find a null sequence $(t_n)_{n\in\N}\subset \R_+$ with $\lim\limits_{n\to \infty}{\omega(t_n)}{t_n^{-\gamma/2}}=0.$
This is true because $\displaystyle\liminf_{t\to 0} \omega(t)t^{-\gamma/2}=0,$ since otherwise we would have $\displaystyle\int_0^T \frac{\omega(s)}{s^{1+\gamma/2}}{
	\rm d}s=\infty$ 
which contradict (\ref{eq 2:Dini-condition}). 
\end{remark}
\medskip

\par\noindent Now let $\Lambda=(0=\lambda_0<\lambda_1<...<\lambda_{n+1}=T)$ be
a uniform subdivision of $[0,T],$ i.e., \[|\Lambda|:=\sup_l|\lambda_{l+1}-\lambda_l|=|\lambda_{k+1}-\lambda_k| \  \text{ for each } \ k=0,1,...,n,\]  and consider a family of sesquilinear  forms $\fra_k:V \times V \to \C$  given by  
\begin{equation*}\label{eq:form-moyen integrale}
\fra_k(u,v):=\frac{1}{\lambda_{k+1}-\lambda_k}
\int_{\lambda_k}^{\lambda_{k+1}}\fra(r;u,v){\rm  d}r,\quad  u,v\in V \ \
\end{equation*}
for each $k=0,1,...,n.$ Remark that $\fra_k$ satisfies (\ref{eq:continuity-nonaut-introduction}) for all $k=0,1,...n.$
Then  $\fra_\Lambda:[0,T]\times V \times V \to \C$ defined for $t\in [\lambda_k,\lambda_{k+1}]$ by
\begin{equation}\label{form: approximation formula2}
\fra_{\Lambda}(t;u,v):=\frac{\lambda_{k+1}-t}{\lambda_{k+1}-\lambda_k}\fra_{k}(u,v)+\frac{t-\lambda_k}{\lambda_{k+1}-\lambda_k}\fra_{k+1}(u,v), \ \ u,v\in V,
\end{equation}
is a non-autonomous closed sesquilinear forms satisfying (\ref{eq:continuity-nonaut-introduction}) with the same constants $\alpha, \beta$ and $M.$ The associated time dependent operator is denoted by
\begin{equation}\label{Operator: approximation formula1}\A_\Lambda(.):[0,T]\to \L(V,V')
\end{equation}
and is given for $t\in [\lambda_k,\lambda_{k+1}]$ by
\begin{equation}\label{Operator: approximation formula2}\
\A_{\Lambda}(t):=\frac{\lambda_{k+1}-t}{\lambda_{k+1}-\lambda_k}\A_{k}+\frac{t-\lambda_k}{\lambda_{k+1}-\lambda_k}\A_{k+1}
\end{equation}
where
 \begin{equation*}\label{eq:op-moyen integrale}
 \A_ku :=\frac{1}{\lambda_{k+1}-\lambda_k}
 \int_{\lambda_k}^{\lambda_{k+1}}\A_m(r)u{\rm  d}r,\quad u\in V, \ k=0,1,...,n.\ \  \end{equation*}
%


%
\medskip 

\noindent In what follows we extend $\omega$ to $[0,2T]$ by setting $\omega(t)=\omega(T)$ for $T\leq t\leq 2T.$
\begin{proposition}\label{Prop: Dini condition for Linear-approximation}
	For all $u,v\in V,\ t,s\in [0,T]$ we have
	\begin{equation}\label{eq:Dini-condition for Linear-approximation}
	|\fra_\Lambda(t,u,v)-\fra_\Lambda(s,u,v)| \le\omega_\Lambda(|t-s|) \Vert u\Vert_{V} \Vert v\Vert_{V_\gamma}
	\end{equation}
	where $\omega_\Lambda:[0,T]\longrightarrow [0,+\infty[$ is defined by
	\[ \omega_\Lambda
	(t):=\left\{%
	\begin{array}{ll}
	\frac{t}{|\Lambda|}\omega(4|\Lambda|)& \hbox{for } 0\leq t\leq 2|\Lambda|,\\
	2\omega(2t) & \hbox{ for } 2|\Lambda|<t\leq T. \\
	\end{array}%
	\right. \]
	Moreover,   $\A_\Lambda(t)-\A_\Lambda(s)\in \L(V,V_\gamma'),$
	\begin{equation}\label{Eq: Dini condition operators}
	\|\A_\Lambda(t)-\A_\Lambda(s)\|_{\L(V,V_\gamma')}\leq \omega_\Lambda(|t-s|)
	\end{equation}
	and \begin{equation}\label{Eq2:Dini condition operators}
	\|\A_\Lambda(t)-\A(t)\|_{\L(V,V_\gamma')}\leq 2\omega(2|\Lambda|)
	\end{equation}
	for each  $t,s\in [0,T].$
\end{proposition}
\begin{proof} Let $u,v\in V$ and $t,s\in [0,T].$  For the proof of (\ref{eq:Dini-condition for Linear-approximation}) we distinguish three cases
	\par\noindent\textit{Case 1:} If $\lambda_k\leq s<t \leq\lambda_{k+1}$ for some fixed $k\in\{0,1,\cdots,n\}.$ Then we obtain, using (\ref{eq 3:Dini-condition}) and the fact that $\omega$ is non-decreasing, that
	\begin{align*}
	|\fra_\Lambda(t,u,v)-\fra_\Lambda(s,u,v)|
	&=\Big\vert\frac{\lambda_{k+1}-t}{\lambda_{k+1}-\lambda_k}\fra_{k}(u,v)+\frac{t-\lambda_k}{\lambda_{k+1}-\lambda_k}\fra_{k+1}(u,v)
	\\&\qquad\qquad\qquad\ \ \ \ -\frac{\lambda_{k+1}-s}{\lambda_{k+1}-\lambda_k}\fra_{k}(u,v)-\frac{s-\lambda_k}{\lambda_{k+1}-\lambda_k}\fra_{k+1}(u,v)\Big\vert
	\\&=\frac{(t-s)}{|\Lambda|}\Big\vert\fra_{k}(u,v)-\fra_{k+1}(u,v)\Big\vert
	\\&\leq \frac{(t-s)}{|\Lambda|}\frac{1}{|\Lambda|}\int_{0}^{|\Lambda|}\mid a(r+\lambda_{k},u,v)-a(r+\lambda_{k+1},u,v)\mid{\rm  d}r
	\\&\leq \frac{(t-s)}{|\Lambda|}\frac{1}{|\Lambda|}\int_{0}^{|\Lambda|}\omega(\lambda_{k+1}-\lambda_{k})\Vert u\Vert_{V} \Vert v\Vert_{V_\gamma}{\rm  d}r
	=\frac{(t-s)}{|\Lambda|}\omega(|\Lambda|)\Vert u\Vert_{V} \Vert v\Vert_{V_\gamma}
	\end{align*}
	\par\noindent \textit{Case 2:} If $\lambda_k\leq s\leq \lambda_{k+1} \leq t\leq\lambda_{k+2},$  then we deduce from \textit{Step 1} that
	\begin{align*}
	|\fra_\Lambda(t,u,v)-\fra_\Lambda(s,u,v)|&\leq
	|\fra_\Lambda(t,u,v)-\fra_\Lambda(\lambda_{k+1},u,v)|+|\fra_\Lambda(\lambda_{k+1},u,v)-\fra_\Lambda(s,u,v)|
	\\&\leq \frac{t-\lambda_{k+1}}{|\Lambda|}\omega(|\Lambda|)\Vert u\Vert_{V} \Vert v\Vert_{V_\gamma}+
	\frac{\lambda_{k+1}-s}{|\Lambda|}\omega(|\Lambda|)\Vert u\Vert_{V} \Vert v\Vert_{V_\gamma}
	\\&=\frac{t-s}{|\Lambda|}\omega(|\Lambda|)\Vert u\Vert_{V} \Vert v\Vert_{V_\gamma}.
	\end{align*}
	\textit{Case 3:}
	If now $\lambda_k\leq s\leq \lambda_{k+1}<\cdots<\lambda_{l}\leq t\leq\lambda_{l+1}.$ Then $\lambda_l-\lambda_{k+1}\leq t-s\leq \lambda_{l+1}-\lambda_{k}$ and thus
	\begin{equation}\label{Eq1 Proof Proposition: Dini-condition for Linear-approximation}|t-s +\lambda_{k+1}-\lambda_{l+1}|\leq |\Lambda|.
	\end{equation} It follows that
	\begin{align*}
	\fra_\Lambda&(t,u,v)-\fra_\Lambda(s,u,v)\\&=
	\frac{\lambda_{l+1}-t}{\lambda_{l+1}-\lambda_l}\fra_{l}(u,v)+\frac{t-\lambda_l}{\lambda_{l+1}-\lambda_l}\fra_{l+1}(u,v)
	-\frac{\lambda_{k+1}-s}{\lambda_{k+1}-\lambda_l}\fra_{k}(u,v)-\frac{s-\lambda_k}{\lambda_{k+1}-\lambda_k}\fra_{k+1}(u,v)
	\\&=\frac{\lambda_{l+1}-t}{|\Lambda|}[\fra_l(u,v)-\fra_k(u.v)]+\frac{t-\lambda_{l}}{|\Lambda|}[\fra_{l+1}(u,v)-\fra_{k+1}(u.v)]
	\\&\qquad+\frac{\lambda_{l+1}-\lambda_{k+1}+s-t}{|\Lambda|}\fra_k(u,v)+\frac{\lambda_{k}-\lambda_{l}+t-s}{|\Lambda|}\fra_{k+1}(u,v)
	\end{align*}
	Because of (\ref{Eq1 Proof Proposition: Dini-condition for Linear-approximation}) and since $\lambda_k-\lambda_l=\lambda_{k+1}-\lambda_{l+1},$ we deduce that
	\begin{align*}
	\mid\fra_\Lambda(t,u,v)-\fra_\Lambda(s,u,v)\mid&\leq \frac{\lambda_{k+1}-t}{|\Lambda|}\omega(\lambda_l-\lambda_k)+\frac{t-\lambda_{k}}{|\Lambda|}\omega(\lambda_{l+1}-\lambda_{k+1})
	\\&\qquad\qquad \ \ \ \ \ \ \ \ +\frac{\mid t-s+\lambda_{l+1}-\lambda_{k+1}\mid}{|\Lambda|}\omega(\lambda_{l+1}-\lambda_{l})
	\\&\leq\omega(\lambda_l-\lambda_k)+\omega(\lambda_{l+1}-\lambda_l)
	\\&\leq 2\omega(2(t-s)).
	\end{align*}
	This completes the proof of (\ref{eq:Dini-condition for Linear-approximation}). 
	Next, (\ref{Eq: Dini condition operators}) follows from (\ref{eq:Dini-condition for Linear-approximation}).
		For the second statement, let $t\in[0,T]$ and let $k\in\{0,1,\cdots,n\}$ be such that $t\in[\lambda_k,\lambda_{k+1}].$ Then
		\begin{align*}
		\A_\Lambda(t)-\A_m(t)&=\frac{\lambda_{k+1}-t}{\lambda_{k+1}-\lambda_k}[\A_k-\A_m(t)]+
		\frac{t-\lambda_{k}}{\lambda_{k+1}-\lambda_k}[\A_{k+1}-\A_m(t)]
		\\&=\frac{\lambda_{k+1}-t}{(\lambda_{k+1}-\lambda_k)^2}\int_{\lambda_k}^{\lambda_{k+1}}[\A_m(r)-\A_m(t)]{\rm d}r+
		\frac{t-\lambda_{k}}{(\lambda_{k+1}-\lambda_k)^2}\int_{\lambda_{k+1}}^{\lambda_{k+2}}[\A_m(r)-\A_m(t)]{\rm d}r.
		\end{align*}
		Then using (\ref{eq 3:Dini-condition}) and the fact that $\omega$ is non-decreasing we obtain
		\begin{align*}
		\|\A_\Lambda(t)-\A_m(t)\|_{\L(V,V_\gamma')}
		&\leq \frac{\lambda_{k+1}-t}{(\lambda_{k+1}-\lambda_k)^2}\int_{\lambda_k}^{\lambda_{k+1}}\omega(t-r){\rm d}r+
		\frac{t-\lambda_{k}}{(\lambda_{k+1}-\lambda_k)^2}\int_{\lambda_{k+1}}^{\lambda_{k+2}}\omega(t-r){\rm d}r
		\\&\leq \omega(|\Lambda|)+\omega(2|\Lambda|)\leq 2\omega(2|\Lambda|),
		\end{align*}
		which proves the claim.
\end{proof}
\noindent Recall that a coercive and bounded form $b:V\times V\to \C$ associated with the operator $B$ on $H$ has the Kato square root property if
\begin{equation}\label{squar root properties} D(B^{1/2})=V.\end{equation}
\noindent We prove in Proposition \ref{square root property for the Linear-approximation} below that $\fra_\Lambda(t,\cdot,\cdot)$ has the square root property for all $t\in[0,T]$ if $\fra_\Lambda(0;\cdot,\cdot)$ has it.  This is essentially based on the abstract result due to Arendt and Monniaux \cite[Proposition 2.5]{Ar-Mo15}. They proved that for two sesquilinear forms $\fra_1,\fra_2:V\times V\to\C$ which satisfies (\ref{eq:continuity-nonaut-introduction}), the form $\fra_1$ has the square root property if and only if $\fra_2$ has it provided that
\[|\fra_1(u,v)-\fra_2(u,v)|\leq c \|u\|_V\|v\|_{V_\gamma} \ u,v\in V\]
for some constant $c>0.$
\begin{proposition}\label{square root property for the Linear-approximation} Assume $\fra(0,.,.)$ has the square root property. Then $\fra_\Lambda(t,.,.)$ has the square root properties for all $t\in [0,T],$ too.
\end{proposition}
\begin{proof} Let $t\in [0,T]$ and let $k\in\{0,1,\cdots,n\}$ be such that $t\in [\lambda_k,\lambda_{k+1}].$ Then Then assumption (\ref{eq 3:Dini-condition}) implies that
	\begin{align*}\mid \fra_\Lambda(t,u,v)-&\fra(0,u,v)\mid\leq\frac{1}{\lambda_{k+1}-\lambda_k}
	\int_{\lambda_k}^{\lambda_{k+1}}\mid\fra(r;u,v)-\fra(0,u,v)\mid{\rm  d}r
	\\&\qquad \qquad \ \ \ +\frac{1}{\lambda_{k+2}-\lambda_{k+1}}
	\int_{\lambda_{k+1}}^{\lambda_{k+2}}\mid\fra(r;u,v)-\fra(0,u,v)\mid{\rm  d}r
	\\&\leq \frac{1}{\lambda_{k+1}-\lambda_k}
	\int_{\lambda_k}^{\lambda_{k+1}}\omega(r)\|u\|_V\|v\|_{V_\gamma}{\rm  d}r
	+\frac{1}{\lambda_{k+2}-\lambda_{k+1}}
	\int_{\lambda_{k+1}}^{\lambda_{k+2}}\omega(r)\|u\|_V\|v\|_{V_\gamma}{\rm  d}r
	\\&\leq 2\sup_{t\in[0,T]}\omega(t)\|u\|_V\|v\|_{V_\gamma}.
	\end{align*}
	Now the claim follows from \cite[Proposition 2.5]{Ar-Mo15}.\end{proof}
\noindent

\noindent Let $\A_\Lambda$ be given by (\ref{Operator: approximation formula2}) and  consider the Cauchy problem
\begin{equation}\label{nCP in V'}
\dot{u}_\Lambda (t)+\A_\Lambda(t)u_\Lambda(t)=f(t) \quad  {a.e.} \ \ \text on \quad [0,T], \ \ u_\Lambda(0)=u_0.
\end{equation}
Next, we use the above results to prove that $\fra_\Lambda$ satisfies all assumption $\bf (H_1)$-$\bf(H_6)$ by taking $d_n=\omega(2|\Lambda|)=\omega(2T/n)$ and $\omega_n(\cdot)=\omega_\Lambda(\cdot).$

\begin{proposition}\label{main result affine approximation} Assume that $\fra$ satisfies (\ref{eq 1:Dini-condition})-(\ref{eq 3:Dini-condition}) and that $\fra(0,\cdot\cdot)$ has the square properties. Then  $\fra_\Lambda$ satisfies assumptions $\bf (H_1)$-$\bf(H_5),$ and satisfies also  $\bf(H_6)$ if moreover (\ref{model contiuity additional}) holds. Furthermore, the solution $u_\Lambda$ of (\ref{nCP in V'}) belongs to $C(0,T;V)$ for each given $u_0\in V$ and $f\in L^2(0,T;H).$
\end{proposition}
\begin{proof} According to Proposition \ref{Prop: Dini condition for Linear-approximation}, $\fra_\Lambda$ satisfies $\bf (H_1)$ and $\bf(H_2).$ By the definition of $\omega_\Lambda$ it follows
	\begin{align*}\int_0^T\frac{\omega_\Lambda(t)}{t^{1+\gamma/2}}{\rm  d}t&
	=\int_0^{2|\Lambda|}\frac{\omega(4|\Lambda|)}{|\Lambda|}t^{-\gamma/2}{\rm  d}t+\int_{2|\Lambda|}^T\frac{\omega(2t)}{t^{1+\gamma/2}}{\rm  d}t
	\\&\leq\const\frac{\omega(4|\Lambda|)}{(4|\Lambda|)^{\gamma/2}}+\const\int_{0}^{2T}\frac{\omega(t)}{t^{1+\gamma/2}}{\rm  d}t
	\\&\leq\const\sup_{t\in[0,T]} \frac{\omega(t)}{t^{\gamma/2}}+\const\int_{0}^{2T}\frac{\omega(t)}{t^{1+\gamma/2}}{\rm  d}t<\infty
	\end{align*}
	which is finite by (\ref{eq 1:Dini-condition}) and (\ref{eq 2:Dini-condition}). Next, it is easy to prove that  \begin{equation}
	\sup_{t\in[0,T]}\frac{\omega_\Lambda(t)}{t^{\gamma/2}}\leq \const\sup_{t\in[0,T]}\frac{\omega(t)}{t^{\gamma/2}}<\infty.
	\end{equation} holds.
	On the other hand, the function $t\mapsto\fra_\Lambda(\cdot,u,v)$ is piecewise $C^1$ for all $u,v\in V$ and  $\fra_\Lambda(t,\cdot,\cdot), t\in[0,T],$ has the Kato square property by Lemma \ref{square root property for the Linear-approximation}. Then the Cauchy problem (\ref{nCP in V'}) has  $L^2-$maximal regularity in $H$ and  $u_\Lambda\in C(0,T;V)$ for each $u_0\in V$ and $f\in L^2(0,T;H)$ \cite[Theorem 4.2]{ADLO14}. Therefore, $\bf (H_3)$ and $\bf (H_4)$ are also satisfied by $\fra.$ Assume now that $(\ref{model contiuity additional})$ holds. Then we obtain that  $\displaystyle{\omega(2|\Lambda|)}{|\Lambda|^{-\gamma/2}}$ and 
	\begin{align*}
	\int_0^{2|\Lambda|} \frac{\omega_\Lambda(t)}{t^{1+\gamma/2}}{\rm  d}t=\frac{\omega(4|\Lambda|)}{|\Lambda|}\int_0^{2|\Lambda|} \frac{{\rm  d}t}{t^{\gamma/2}}=2\frac{\omega(4|\Lambda|)}{(2|\Lambda|)^{\gamma/2}}
	\end{align*}
converge to $0$ as $|\Lambda|\lra 0.$ Thus the fact that $\omega$ is non-decreasing complete the proof.
	
	\end{proof}
  The next provides in particular an alternative proof of some results in \cite{Ar-Mo15}.
\begin{corollary}\label{corollary: estimation uniform approx affine} Assume that $\fra$ satisfies (\ref{eq 1:Dini-condition})-(\ref{eq 3:Dini-condition}) and that $\fra(0,\cdot\cdot)$ has the square properties. Then (\ref{nCP in V'}) has $L^2$-maximal regularity in $H$ and for each $u_0\in V$ and $f\in L^2(0,T;H)$ the solution $(u_\Lambda)_\Lambda$ converges weakly in $MR_2(V,H)$ as $|\Lambda|\lra 0,$ and $u:=w-\lim\limits_{|\Lambda|\to 0} u_\Lambda$ satisfies (\ref{Abstract Cauchy problem 0}). If moreover, $(\ref{model contiuity additional})$ holds then $u_\Lambda\lra u$ strongly in $MR_2(V,H)\cap C(0,T;V)$  and uniformly on $(u_0,f)$ as $|\Lambda|\lra 0.$ Further, the following estimates 
	\begin{equation*}
	\|u_\Lambda-u\|_{MR_2(V,H)}\leq \textbf{c}\left[\big(1+\frac{1}{|\Lambda|^{\gamma/2}}\big)d_n
	+\int_0^{2|\Lambda|}\frac{\omega(t)}{t^{1+\gamma/2}}{\rm d}t
	\right]\Big[\|f\|_{L^2(0,T;H)}+\|u_0\|_V.\Big]
	\end{equation*}
	and 
	\begin{equation*}
	\|u_\Lambda-u\|_{C(0,T;V)}\leq \textbf{c}\left[\big(1+\frac{1}{|\Lambda|^{\gamma/2}}\big)d_n
	+\int_0^{2|\Lambda|}\frac{\omega(t)}{t^{1+\gamma/2}}{\rm d}t
	\right]\Big[\|f\|_{L^2(0,T;H)}+\|u_0\|_V.\Big]
	\end{equation*}
	holds.
\end{corollary}
\begin{proof} The proof follows from Theorem (\ref{convergence uniform}), Proposition \ref{main result affine approximation} and Corollary \ref{corollary continuity of solution}
\end{proof}



\end{document}